\documentclass{amsart}
\usepackage{amsmath,amssymb,latexsym,amscd,color,array,euscript,amsthm,amsfonts}
\usepackage[margin=0.75in]{geometry}
\usepackage{tikz}
\usetikzlibrary{arrows}

\input xy
\xyoption{2cell}
\xyoption{all}
 
\newtheorem{theorem}{Theorem}

\newtheorem{corollary}[theorem]{Corollary}
\newtheorem{definition}[theorem]{Definition}
\newtheorem{example}[theorem]{Example}
\newtheorem{lemma}[theorem]{Lemma}

\newtheorem{remark}[theorem]{Remark}

\newcommand{\cF}{\mathcal{F}}

\newcommand{\cP}{\mathcal{P}}

\newcommand{\cT}{\mathcal{T}}

\newcommand{\ZZ}{\mathbb{Z}}

\renewcommand{\max}{\text{max}}

\newcommand{\supp}{\text{supp}}

\renewcommand{\eqref}[1]{{\rm (\ref{#1})}}

\title{Proof of the Kontsevich Non-Commutative Cluster Positivity Conjecture}
\author{Dylan Rupel}
\address{\noindent Department of Mathematics, University of Oregon,
 Eugene, OR 97403}
\email{drupel@uoregon.edu}

\begin{document}
\begin{abstract}
 We extend the Lee-Schiffler Dyck path model to give a proof of the Kontsevich non-commutative cluster positivity conjecture with unequal parameters.
\end{abstract}
\maketitle

Let $k$ be any field of characteristic zero.  For any $r\in\ZZ_{>0}$, consider the following $k$-linear automorphism of the skew-field $K=k(x,y)$ of rational functions in non-commutative variables $x$ and $y$:
\[F_r: (x,y)\mapsto (xyx^{-1},(1+y^r)x^{-1}).\]
\noindent Our main result is the following.
\begin{theorem}[Kontsevich conjecture]\label{kont_conj}
 For any $r_1,r_2\in\ZZ_{>0}$ and any $k\ge0$, the elements $x_k=\underbrace{F_{r_1}F_{r_2}F_{r_1}\cdots}_k(x)$ are given by non-commutative Laurent polynomials in $x$ and $y$ with non-negative integer coefficients.
\end{theorem}
\begin{remark}
 Using a symmetry argument, Theorem~\ref{kont_conj} implies an analogous statement for $y_k=\underbrace{F_{r_1}F_{r_2}F_{r_1}\cdots}_k(y)$.
\end{remark}
 
\noindent The Laurentness of these expressions was established by Usnich \cite{u} for $r_1=r_2$ and by Berenstein-Retakh \cite{br} for general $r_1,r_2$.  The positivity was shown by Di Francesco-Kedem \cite{dk} for $r_1r_2=4$ and by Lee-Schiffler \cite{ls} for $r_1=r_2$.  We follow the Lee-Schiffler approach in this note.

Fix integers $r_1,r_2\in\ZZ_{>0}$.  Our proof will make use of two-parameter Chebyshev polynomials $U_{k,j}$, $k,j\in \ZZ$, defined recursively by: 
$U_{-1,j}=0$, $\ U_{0,j}=1$, $\ U_{k+1,j+1}=r_jU_{k,j}-U_{k-1,j-1}$, where $r_j=\begin{cases} r_1, & \text{ if $j$ is odd;}\\ r_2, & \text{ if $j$ is even.}\end{cases}$\\ 
\noindent From now on we will work under the assumption $r_1r_2\ge5$.  The cases $r_1r_2\in\{1,4\}$ were settled in \cite{u} and \cite{dk} and the remaining cases $r_1r_2\in\{2,3\}$ are given explicitly at http://pages.uoregon.edu/drupel/dyck\_examples.pdf.

Fix $n\ge2$.  Consider the rectangle $R_n\subset\ZZ^2$ with corner vertices $(0,0)$ and $(U_{n-3,1}-U_{n-4,2},U_{n-4,2})$.  When $R_n$ lies in the first quadrant, a \emph{Dyck path} is a lattice path in $R_n$ starting at $(0,0)$ and taking North or East steps to end at $(U_{n-3,1}-U_{n-4,2},U_{n-4,2})$ such that the path never crosses the main diagonal of $R_n$ and the slope of each subpath beginning at $(0,0)$ does not exceed the slope of the main diagonal.  Here we consider a vertical edge to have slope $\infty$.  We modify this definition slightly when $R_n$ lies in the second quadrant by replacing the East step with a diagonal $(-1,1)$-upstep and considering vertical edges to have slope $-\infty$.  When $n=2$, $R_n$ lies in the fourth quadrant and we use a diagonal $(1,-1)$-downstep.  We will call a Dyck path \emph{maximal} if no subpath of another Dyck path lies closer to the main diagonal.  Write $D_n$ for the maximal Dyck path in $R_n$.  The next Lemma follows by induction from the definitions.

\begin{lemma}\label{Dyck_rec} Denote $\epsilon_k:=\max\{0,2-r_{k-1}\}$, $\delta_k:=\epsilon_k+2\epsilon_{k-1}+1$ for $k\in \ZZ$. 
 Suppose $k-\delta_k\ge 4$.  Then the Dyck path $D_k$ consists of $r_{k-\epsilon_{k-1}}-\delta_k+1$ copies of $D_{k-1-\epsilon_{k-1}}$ followed by a copy of $D_{k-1-\epsilon_{k-1}}$ with its first $D_{k-1-\delta_k}$ removed.
\end{lemma}

Let $U_n=max\{|U_{n-3,1}|,|U_{n-4,2}|\}$ be the number of edges in $D_n=(\omega_0,\alpha_1,\omega_1,\alpha_2,\ldots,\alpha_{U_n},\omega_{U_n})$, where the vertices of $D_n$ are labeled by $\omega_0,\omega_1,\ldots,\omega_{U_n}$ and $\alpha_i$ is the edge connecting $\omega_{i-1}$ and $\omega_i$.  Let $i_1,\ldots,i_{U_{n-4,2}}$ denote the increasing sequence so that $\alpha_{i_j}$ makes an upward step.  We will write $\nu_0,\ldots,\nu_{U_{n-4,2}}$ for the sequence of vertices satisfying $\nu_0=(0,0)$ and $\nu_j=\omega_{i_j}$.

\begin{definition}\label{colors}
 For $i<j$ denote by $s_{ij}$ the slope of the line from $\nu_i$ to $\nu_j$ and by $s$ the slope of the main diagonal of $R_n$.  For $0\le i<k\le U_{n-4,2}$ let $\alpha(i,k)$ be the subpath of $D_n$ from $\nu_i$ to $\nu_k$ labeled/colored as follows:
 \begin{enumerate}
  \item If $s_{it}\le s_n$ for all $t$ with $i<t\le k$, then $\alpha(i,k)$ is called a \emph{Dyck prefix} (blue).
  \item If $s_{it}>s_n$ for some $t$ with $i<t\le k$, then
   \begin{enumerate}
    \item if the \emph{smallest} such $t$ is of the form $i+U_{m,2}-wU_{m-1-\epsilon_{m-1},2}$ for some integers $1\le m\le n-4$ and $1\le w<r_{m-\epsilon_{m-1}}-\delta_m$, then $\alpha(i,k)$ is called an $(m,w)$-Dyck suffix (green).
    \item otherwise, $\alpha(i,k)$ is called a \emph{short suffix} (red).
   \end{enumerate}
 \end{enumerate}
\end{definition}

Write $\cP(D_n)=\{\alpha(i,k):0\le i<k\le U_{n-4,2}\}\cup\{\alpha_1,\ldots,\alpha_{U_n}\}$ for the set of admissible subpaths of $D_n$.  For $\beta\subset\cP(D_n)$ we define the support $\supp(\beta)\subset D_n$ in the natural way.  We will use the term \emph{hook} for the supports of the subpaths $\alpha(k,k+1)$.  It will be convenient to refer to a hook as type 1, 2, or 3 depending on whether the horizontal displacement from the bottom to the top of the hook is $r_2-1$, $r_2-2$, or $r_2-3$, respectively.  

Call $\beta\subset\cP(D_n)$ an \emph{overlapping} collection if there exists either $\alpha(i,k),\alpha(i',k')\in\beta$ which share a vertex or $\alpha_j,\alpha(i,k)\in\beta$ with $\alpha_j\in\alpha(i,k)$.  We will need the following $K$-valued weightings on non-overlapping collections.

\begin{definition}
 Write $\varepsilon_i=\begin{cases} 1 & \text{if $\alpha_i$ is vertical;}\\ 0 & \text{otherwise.}\end{cases}$  For each non-overlapping collection $\beta\subset\cP(D_n)$ define
 \begin{equation*}
  \beta_{[i]}=\begin{cases}
   y^{r_1-\varepsilon_i}x^{-1}, & \text{ if $\alpha_i\notin\supp(\beta)$;}\\
   y^{-\varepsilon_i}x^{-1}, & \text{ if $\alpha_i\in\beta$ and $\alpha_i$ is not diagonal;}\\
   x^1y^{-1}x^{-1}, & \text{ if $\alpha_i\in\beta$ and $\alpha_i$ is diagonal with an upstep;}\\
   x^0y^1, & \text{ if $\alpha_i\in\beta$ and $\alpha_i$ is diagonal with a downstep;}\\
   x^0y^0, & \text{ if $\alpha_i\in\alpha(j,k)\in\beta$ is horizontal;}\\
   x^hy^{-1}x^{-1}, & \text{ if $\alpha_i\in\alpha(j,k)\in\beta$ is the last edge of a hook of type h.}\\
  \end{cases}
 \end{equation*}
\end{definition}
\noindent We have the following refinement of Theorem \ref{kont_conj}.  

\begin{theorem}\label{main}
 Suppose $r_1,r_2\in\ZZ_{>0}$.  Write $q=xyx^{-1}y^{-1}$.  Then for $n\ge2$ we have $x_{n-1}=\sum\limits_{\beta\in\cF(D_n)} q\prod\limits_{i=1}^{U_n}\beta_{[i]}$, where the product is taken in the natural order and the sum ranges over the set $\cF(D_n)$ of non-overlapping collections $\beta\subset\cP(D_n)$ subject to the conditions:
 \begin{itemize}
  \item[$C1$:] if $\alpha_i$ is diagonal, then $\alpha_i$ is supported on $\beta$;
  \item[$C2$:] if $\alpha(i,k)\in\beta$ is a short suffix, then the preceding non-diagonal edge of $\nu_i$ is supported on $\beta$;
  \item[$C3$:] if $\alpha(i,k)\in\beta$ is an $(m,w)$-Dyck suffix, then at least one of the preceding $U_{m-1,1}-wU_{m-2-\epsilon_{m-1},1}$ non-diagonal edges of $\nu_i$ is supported on $\beta$.
 \end{itemize}
\end{theorem}

\begin{example}
 \begin{tabular}{cl}
  \raisebox{.2cm}{For $r_1=2$, $r_2=3$, $n=5$ we have $U_{2,1}=5$, $U_{1,2}=2$ and so $R_5$ and $D_5$ are given by:} &
  \begin{tikzpicture}
   \draw[step=0.25cm,color=gray] (0,0) grid (0.75,0.5);
   \draw[color=gray] (0,0) -- (0.75,0.5);
   \draw[fill=black] (0,0) circle (1.1pt);
   \draw[fill=black] (0.25,0) circle (1.1pt);
   \draw[fill=black] (0.5,0) circle (1.1pt);
   \draw[fill=black] (0.5,0.25) circle (1.1pt);
   \draw[fill=black] (0.75,0.25) circle (1.1pt);
   \draw[fill=black] (0.75,0.5) circle (1.1pt);
  \end{tikzpicture}.\\
 \end{tabular}
 \noindent We have the following expression for $x_4$:
 \begin{align*}
  x_4=&qxy^{-1}xy^{-1}x^{-1}+qxy^{-1}x^{-1}(1+y^2)x^{-1}(1+y^2)y^{-1}x^{-1}+q(1+y^2)x^{-1}(1+y^2)x^{-1}y^{-1}xy^{-1}x^{-1}+\\
  &+q(1+y^2)x^{-1}(1+y^2)x^{-1}(1+y^2)y^{-1}x^{-1}(1+y^2)x^{-1}(1+y^2)y^{-1}x^{-1},
 \end{align*}
 where a factor of $1+y^2$ indicates an edge which may be either included in or excluded from the corresponding admissible collection of labeled/colored subpaths.  We present several examples for $r_1r_2=5$, enumerating all admissible collections with their monomials, at http://pages.uoregon.edu/drupel/dyck\_examples.pdf.
\end{example}

\begin{proof}[Proof of Theorem~\ref{main}:]
We divide the proof into a series of lemmas.  First we make the following definitions.
\begin{definition}
 Define the set $\tilde{\cF}(D_n)$ of non-overlapping collections $\beta\subset\cP(D_n)$ subject to conditions $C1$ and $C2$.
 Define $\cT^{\ge u}(D_n)\subset\tilde{\cF}(D_n)$ to consist of those $\beta$ satisfying the following condition only for $m\ge u$:
 \begin{itemize}
  \item[$C3^{op}$:] there exists integers $i,k,w,m$ such that $\alpha(i,k)\in\beta$ is an $(m,w)$-Dyck suffix and none of the preceding $U_{m-1,1}-wU_{m-2-\epsilon_{m-1},1}$ non-diagonal edges of $\nu_i$ are supported on $\beta$.
 \end{itemize}
\end{definition}
\begin{lemma}\label{filt_rest}
 If $m\ge n-3$, there do not exist $i,w$ $(1\le w< r_{m-\epsilon_{m-1}}-\delta_m)$ so that $\min\{t:i<t\le U_{n-4,2}, s_{i,t}>s\}$ is of the form $i+U_{m,2}-wU_{m-1-\epsilon_{m-1},2}$.  In particular, for any $n\ge2$, the set $\cT^{\ge n-3}(D_n)$ is empty.
\end{lemma}
\begin{proof}
 We assume $\epsilon_{m-1}=0$; the case $\epsilon_{m-1}>0$ follows from this one.  Since $w<r_m-1-\epsilon_m$, we have 
 \[U_{m,2}-wU_{m-1,2}\ge U_{m,2}-r_mU_{m-1,2}+(2+\epsilon_m)U_{m-1,2}=(2+\epsilon_m)U_{m-1,2}-U_{m-2,2}\ge U_{m-k,2}, \text{ for $k\ge1$.}\] 
 Now if $m\ge n-3$ and $\tau:=\min\{t:i<t\le U_{n-4,2}, s_{i,t}>s\}=i+U_{m,2}-wU_{m-1,2}$, then $\tau\ge i+U_{n-4,2}$.  But this contradicts $\nu_{U_{n-4,2}}$ being the highest labeled vertex in $D_n$.
\end{proof}

\noindent Let $z_0=x_0=x$ and for $n\ge2$ write $z_{n-1}=\sum\limits_{\beta\in\tilde{\cF}(D_n)} q\prod\limits_{i=1}^{U_n} \beta_{[i]}$.  
For each integer $\ell$ we will use a parenthesized exponent ${}^{(\ell)}$ to denote a quantity with each $r_k$ replaced by $r_{k+\ell}$.  

\begin{lemma}\label{z_rec}
 Suppose $n\ge2$. Then $z^{(1)}_n=F_{r_2}(z_{n-1})+\sum\limits_{\beta\in\cT^{\ge1}(D^{(1)}_{n+1})\setminus\cT^{\ge2}(D^{(1)}_{n+1})} q\prod\limits_{i=1}^{U^{(1)}_{n+1}} \beta_{[i]}$.
\end{lemma}
\begin{proof}
 This follows from a study of how the $(1+y^{r_2})^{-1}$ terms cancel in $F_{r_2}(z_{n-1})$.  In particular, we make the following observations.  The sum of the weights of a colored hook and the corresponding full hook of uncolored edges gives rise to a Laurent monomial under $F_{r_2}$.  An edge $\alpha$ in the support of $\beta$ gives rise to a colored hook of type 1, 2, or 3 corresponding to the edge $\alpha$ being horizontal, vertical not followed by a diagonal, or vertical followed by a diagonal, respectively.  A missing edge $\alpha$ gives rise to all collections of uncolored edges in a hook of type 1, 2, or 3 corresponding to the edge $\alpha$ being horizontal, vertical not followed by a diagonal, or vertical followed by a diagonal, respectively.  

 Now consider an uncolored hook with a missing horizontal edge, followed by $d$ included horizontal edges, and then an included vertical edge.  Under $F_{r_2}$ the weight of this configuration gives rise to the weights of all collections of horizontal edges in a hook of type 1 with an included vertical edge followed by $d$ colored hooks of type 1 and then a colored hook of type 2.  The sum is accounting for the included vertical edge in this case.
\end{proof}
\noindent In the following Lemma we consider a $D_3$ with its first $D_2$ removed as a single vertical edge and for $\epsilon_3=1$ we consider a $D_4$ with its first $D_2$ removed as a vertical edge followed by a $(-1,1)-$diagonal edge.
\begin{lemma}\label{GREEN_rec}\mbox{}
 \begin{enumerate}
  \item\label{step1} Suppose $k-\epsilon_{k-1}\ge5$.  Then the weight of a missing $D_{k-2}$ with its first $D_{k-3-\epsilon_{k-3}}$ removed followed by a colored $D_k$ simplifies to the weight of a colored $D_{k-1-\epsilon_{k-1}}$.
  \item\label{step2} Suppose $k-\epsilon_{k-1}\ge5$.  Then the weight of a missing $D_{k-2}$ followed by a colored $D_{k-1-\epsilon_{k-1}}$ simplifies to the weight of a missing $D_{k-2}$ with its first $D_{k-3-\epsilon_{k-3}}$ removed.
  \item\label{step3} Suppose $m-\delta_m\ge0$.  Then for $1\le w< r_{m-\epsilon_{m-1}}-\delta_m$, the weight of an $(m,w)$-Dyck suffix preceded by $U_{m-1,1}-wU_{m-2-\epsilon_{m-1},1}$ missing non-diagonal edges is equal to the weight of an $(m,w+1)$-Dyck suffix preceded by $U_{m-1,1}-(w+1)U_{m-2-\epsilon_{m-1},1}$ missing non-diagonal edges.
 \end{enumerate}
\end{lemma}
\begin{proof}
 Parts~\eqref{step1} and \eqref{step2} follow from a simultaneous induction using Lemma~\ref{Dyck_rec} in the induction step.  Part~\eqref{step3} follows from \eqref{step1}, \eqref{step2}, and Lemma~\ref{Dyck_rec}.
\end{proof}

\begin{corollary}
 Suppose $m-\delta_m\ge0$.  Then for $1\le w< r_{m-\epsilon_{m-1}}-\delta_m$, the weight of an $(m,w)$-Dyck suffix preceded by $U_{m-1,1}-wU_{m-2-\epsilon_{m-1},1}$ missing non-diagonal edges is equal to $q^{-1}$.
\end{corollary}
\begin{proof}
 We work by induction, the case $m-\delta_m=0$ is easy to check by hand.  It follows from Lemma~\ref{Dyck_rec} that the hook sequences of an $(m,r_{m-\epsilon_{m-1}}-\delta_m)$-Dyck suffix and an $(m-1,1)$-Dyck suffix are the same.  Then one easily checks that $U_{m-1,1}-(r_{m-\epsilon_{m-1}}-\delta_m)U_{m-2-\epsilon_{m-1},1}=U_{m-2,1}-U_{m-3-\epsilon_{m-2},1}$,
 the case $\epsilon_{m-1}>0$ following from the case $\epsilon_{m-1}=0$.  The result now follows by induction using Lemma~\ref{GREEN_rec}.\ref{step3}.
\end{proof}

\noindent We remind that a parenthesized exponent ${}^{(\ell)}$ denotes a quantity with each $r_k$ replaced by $r_{k+\ell}$.  In particular, note that $F_{r_2}(x_k)=x^{(1)}_{k+1}$.
\begin{lemma}\label{F_rec}
 Let $u\ge1$ and $n\ge u+4$. Then
\begin{equation*}
 F_{r_2}\left(\sum\limits_{\beta\in\cT^{\ge u}(D_n)\setminus\cT^{\ge u+1}(D_n)} q\prod_{i=1}^{U_n} \beta_{[i]}\right)
 =\sum\limits_{\beta\in\cT^{\ge u+1}(D^{(1)}_{n+1})\setminus\cT^{\ge u+2}(D^{(1)}_{n+1})} q\prod_{i=1}^{U^{(1)}_{n+1}} \beta_{[i]}.
\end{equation*}
\end{lemma}
\begin{proof}
 The proof follows by simultaneous induction with Lemma~\ref{x_comp}.  We will assume $n=u+4$, the case $n>u+4$ follows from this one using a similar argument.  Also we restrict to the case $\epsilon_{n-1}=0$, the case $\epsilon_{n-1}>0$ follows by a similar argument.  

 From Lemma~\ref{Dyck_rec}, we can see that $D_n$ begins with $w$ copies of $D_{n-1}$, $1\le w < r_n-1-\epsilon_n$, and the vertex $\nu_{wU_{n-5,2}}$ is the ending vertex of the last $D_{n-1}$.  Now $\alpha(wU_{n-5,2},U_{n-4,2})$ is the only $(n-4,w)$-Dyck suffix of $D_n$ and so $\beta\in\cT^{\ge n-4}(D_n)$ implies $\alpha(wU_{n-5,2},U_{n-4,2})\in\beta$ and none of the preceding $U_{n-5,1}-wU_{n-6,1}$ non-diagonal edges are contained in $\beta$.  Note that $wU_{n-4,1}-U_{n-5,1}+wU_{n-6,1}=r_2wU_{n-5,2}-U_{n-5,1}$ and so the lowest vertex of these missing edges is $\omega_{r_2wU_{n-5,2}-U_{n-5,1}}$.  Then Lemma~\ref{Dyck_rec} implies the subpath of $D_n$ from $\omega_0$ to $\omega_{r_2wU_{n-5,2}-U_{n-5,1}}$ consists of $w-1$ copies of $D_{n-1}$, followed by $r_{n-1}-1$ copies of $D_{n-2}$, and then $w-1$ copies of $D_{n-3}$.  We will define $j_i$ for $0\le i\le 2w+r_{n-1}-3$ so that the $\nu_{j_i}$ are the endpoints of these copies.  Any subpath $\alpha(i,k)$ can be decomposed as $\alpha(i,j_e),\alpha(j_e,j_{e+1}),\ldots,\alpha(j_{e+\ell},k)$ where all but the first are Dyck prefixes.  It is easy to see that $\alpha(i,j_e)$ has the same label/color as $\alpha(i,k)$ and if $\alpha(i,k)$ was an $(m,w')$-Dyck suffix then so is $\alpha(i,j_e)$.
 
 Combining the above considerations we see that $\displaystyle\sum\limits_{\beta\in\cT^{\ge u}(D_n)\setminus\cT^{\ge u+1}(D_n)} q\prod_{i=1}^{U_n} \beta_{[i]}$ can be rewritten as:
 \begin{align*}
  &\sum\limits_{w=1}^{r_n-2-\epsilon_n} q\left(\sum\limits_{\beta\in\cF(D_{n-1})}\prod_{i=1}^{U_{n-1}}\beta_{[i]}\right)^{w-1}\left(\sum\limits_{\beta\in\cF(D_{n-2})}\prod_{i=1}^{U_{n-2}}\beta_{[i]}\right)^{r_{n-1}-1}\left(\sum\limits_{\beta\in\cF(D_{n-3})}\prod_{i=1}^{U_{n-3}}\beta_{[i]}\right)^{w-1}q^{-1}\\
  &=\sum\limits_{w=1}^{r_n-2-\epsilon_n} q\left(q^{-1}x_{n-2}\right)^{w-1}\left(q^{-1}x_{n-3}\right)^{r_{n-1}-1}\left(q^{-1}x_{n-4}\right)^{w-1}q^{-1},
 \end{align*}
 where the equality follows from Lemma~\ref{x_comp}. Applying $F_{r_2}$ and noting that $F_{r_2}(q)=q$ completes the proof.
\end{proof}

\begin{lemma}\label{x_comp}
 Suppose $n\ge3$. Then
\begin{equation}\label{x_comp_eq}
 x_{n-1}=z_{n-1}-\sum\limits_{m=5}^n \underbrace{F_{r_1}F_{r_2}F_{r_1}\cdots}_{n-m}\left(\sum\limits_{\beta\in\cT^{\ge1}(D^{(m-n)}_m)\setminus\cT^{\ge2}(D^{(m-n)}_m)} q\prod_{i=1}^{U^{(m-n)}_m} \beta_{[i]}\right)=\sum\limits_{\beta\in\cF(D_n)} q\prod_{i=1}^{U_n} \beta_{[i]}.
\end{equation}
\end{lemma}
\begin{proof}
 This follows from simultaneous induction with Lemma~\ref{F_rec} as in the proof of \cite[Lemma 20]{ls}.
\end{proof}
\end{proof}

\noindent\textbf{Acknowledgements.} This project was initiated while the author was visiting the new Center for Mathematics at Notre Dame in June of 2011.  We would like to thank CMND for their hospitality.

\end{document}